\documentclass[11pt]{article}
\usepackage{fullpage,amsthm,amssymb,tikz,amsmath,verbatim,moresize}
\usepackage[font=footnotesize]{caption} 
\usepackage{moresize}
\def\bm{\boldsymbol}
\def\mad{\textrm{mad}}

\def\diam{\textrm{diam}}
\def\ch{{\mathrm{ch}}}

\def\G{\mathcal{G}}
\def\R{\mathcal{R}}
\def\uphr{\upharpoonright}

\newtheorem{lem}{Lemma}
\newtheorem{thm}{Theorem}
\newtheorem{thmA}{Theorem}

\newtheorem*{mad-lem}{Mad Lemma}
\newtheorem*{keylem}{Key Lemma}

\RequirePackage{marginnote,hyperref}
\newcommand{\aside}[1]{\marginnote{\scriptsize{#1}}[0cm]}
\newcommand{\aaside}[2]{\marginnote{\scriptsize{#1}}[#2]}
\newcommand\Emph[1]{\emph{#1}\aside{#1}}
\newcommand\EmphE[2]{\emph{#1}\aaside{#1}{#2}}
\author{Daniel W. Cranston\thanks{%
Department of Computer Science, Virginia Commonwealth
University, Richmond, VA, USA;
\texttt{dcranston@vcu.edu}
}}
\begin{document}
\title{List-Recoloring of Sparse Graphs}
\maketitle
\abstract{
Fix a graph $G$, a list-assignment $L$ for $G$, and $L$-colorings $\alpha$ and
$\beta$.  
An $L$-recoloring sequence,
starting from $\alpha$, 
recolors a
single vertex at each step, so
that each resulting intermediate coloring is a proper $L$-coloring.  An
$L$-recoloring sequence transforms $\alpha$ to $\beta$ if its initial coloring
is $\alpha$ and its final coloring is $\beta$.  
We prove there exists an $L$-recoloring sequence that transforms
$\alpha$ to $\beta$ and recolors each vertex at most a constant number of times
if (i) $G$ is triangle-free and planar and $L$ is a 7-assignment, or (ii)
$\mad(G)<17/5$ and $L$ is a 6-assignment or (iii) $\mad(G)<22/9$ and $L$ is a
4-assignment.  Parts (i) and (ii) confirm conjectures of Dvo\v{r}\'{a}k and Feghali.
}

\section{Introduction}
A proper $k$-coloring of a graph $G$ assigns each vertex $v$ of $G$ a
``color'' from $\{1,\ldots,k\}$ so that the endpoints of each edge get distinct
colors.  A \emph{recoloring step} changes the color of a single vertex, so that
the resulting coloring is also proper.  Given two proper $k$-colorings of $G$,
say $\alpha$ and $\beta$, we want to know if we can transform $\alpha$ to $\beta$ 
by a sequence of recoloring steps, again requiring that each intermediate coloring
is a proper $k$-coloring.  If we can, then we might ask for the ``distance''
from $\alpha$ to $\beta$, the length of a shortest sequence of recoloring steps
that transforms $\alpha$ to $\beta$?  And, more generally, what is the maximum
distance between any pair $\alpha$ and $\beta$?
Such questions have been studied extensively, under the name Glauber dynamics, due to
their applications in statistical physics; for example, see~\cite{FV-survey,
SODA-best} and their
references.

Before going further, we will rephrase the problem in a bit more generality.
For brevity, we defer many well-known definitions to the end of
Section~\ref{prelims-sec}.  Rather than considering $k$-colorings of $G$, we fix a
list-assignment $L$ such that $|L(v)|=k$ for all vertices $v$, and we consider
$L$-colorings.  Given a proper $L$-coloring of $G$, an \emph{$L$-recoloring step}
changes the color of a single vertex, so that the resulting coloring is also a
proper $L$-coloring.  Now we ask the same questions as above.  

Given $L$-colorings $\alpha$ and
$\beta$, can we transform $\alpha$ to $\beta$ by a sequence of $L$-recoloring steps,
requiring that each intermediate coloring is a proper $L$-coloring?  If so, what
is the length of a shortest such recoloring sequence, the ``distance'' from
$\alpha$ to $\beta$?  And what is the maximum distance over all pairs $\alpha$
and $\beta$?  Formally, the \emph{$L$-recoloring graph of $G$}, denoted
$\R_L(G)$, has as its vertices all $L$-colorings of $G$, and two vertices of
the $L$-recoloring graph are adjacent if the $L$-colorings differ on a single
vertex of $G$.  That is, what is the diameter of $\R_L(G)$?
We can also go further and ask about the maximum over all
$k$-assignments $L$.  

Intuitively, all pairs of $L$-colorings should be ``near'' to each other when
$k$ is sufficiently large, relative to $G$.  But what does near mean?  
Rather than focusing on a single graph $G$, we typically consider families $\G$
of graphs $G$ and study the maximum of $\diam(\R_L(G))$ over all $n$-vertex
graphs $G\in \G$ and all $k$-assignments $L$ for $G$.  Specifically, we study how
this maximum grows with $n$.
It is very possible that
not a single vertex receives the same color under $\alpha$ and $\beta$, i.e.,
$\alpha(v)\ne \beta(v)$ for all vertices $v$.  In that case, clearly
$\diam(R_L(G))\ge n$.  When $G$ is sufficiently sparse, relative to $k$, we can
often show that $\diam(R_L(G))=O(n)$.

As an easy example, suppose that $k\ge 2\Delta(G)+1$.  One by one, we recolor
each vertex $v$ with a color that is not currently used on its open
neighborhood, and also not used on its open neighborhood under $\beta$.  After this,
we can simply recolor each vertex $v$ with $\beta(v)$.  This shows that 
$\diam(\R_L(G))\le 2n$.   However, the requirement $|L(v)|\ge 2\Delta(G)+1$ is a
very strong hypothesis.  For sparser families of graphs, often much weaker
hypotheses on $L$ suffice to guarantee that $\diam(\R_L(G))=O(n)$.  In this
paper, we focus on planar graphs with sufficiently large girth, as well as the
more general class of graphs $G$ with bounded maximum average degree, which we
denote by $\mad(G)$.  An \emph{$L$-recoloring sequence} is simply a sequence of
$L$-recoloring steps, starting from some specified $L$-coloring.  
We prove the following three main results.

\begin{thmA}
Let $G$ be a planar triangle-free graph.  Fix a
7-assignment $L$ for $G$ and $L$-colorings $\alpha$ and $\beta$.  There exists
an $L$-recoloring sequence that transforms $\alpha$ into $\beta$ such that each
vertex is recolored at most 30 times.
\label{thm1A}
\end{thmA}

\begin{thmA}
Fix a graph $G$ with $\mad(G)<17/5$, a 
6-assignment $L$ for $G$, and $L$-colorings $\alpha$ and $\beta$.  There exists
an $L$-recoloring sequence that transforms $\alpha$ into $\beta$ such that each
vertex is recolored at most 12 times.  In particular, this holds for every
planar graph of girth at least 5.
\label{thm2A}
\end{thmA}

\begin{thmA}
Fix a graph $G$ with $\mad(G)<22/9$, a 
4-assignment $L$ for $G$, and $L$-colorings $\alpha$ and $\beta$.  There exists
an $L$-recoloring sequence that transforms $\alpha$ into $\beta$ such that each
vertex is recolored at most 14 times. In particular, this holds for every planar
graph of girth at least 11.
\label{thm3A}
\end{thmA}

Theorems~\ref{thm1A} and~\ref{thm2A} confirm conjectures of Dvo\v{r}\'{a}k and
Feghali~\cite{DF2}.  Before discussing the proofs of
Theorems~\ref{thm1A}--\ref{thm3A}, we say more about the history of recoloring
and related problems.

It is very natural to want to transform one instance of something to another
instance, by a sequence of small steps, maintaining a valid instance at each
step.  We can apply this paradigm to
colorings, dominating sets, independent sets, perfect matchings, spanning
trees, or solutions to a
SAT problem, to name a few.  Once we have specified the type of instance we
want, then we must specify the allowable modifications to move from one instance
to the next.  For SAT problems, we typically restrict to flipping the true/false
value of a single variable.  

But for some problems, the ``best'' choice of
allowable modification is less clear.  When moving from one (list-)coloring to the
next, we might restrict to recoloring a single vertex; or we might allow the
more powerful notion of a Kempe swap.  When moving from one independent set $I$
to another, we might allow replacing any vertex in $I$ with any vertex not in
$I$.  Or we might require that the replaced vertex be adjacent to the vertex
replacing it (so-called, token sliding).  All of these choices give rise to
interesting problems.  And this general area of study is known as
\emph{reconfiguration}.  An early reference on this topic is~\cite{vdH-survey}.
For a more recent survey, see~\cite{nishimura}.

Much work has focused on coloring reconfiguration.  Recently, \cite{BBFHMP}
completed the characterization of pairs $(g,k)$ such that for all planar graphs
$G$ of girth at least $g$, every $k$-coloring of $G$ can be transformed to
every other.  Similar problems have been studied for graphs
of bounded maximum average degree~\cite{feghali-JCTB}, graphs of bounded
treewidth~\cite{BB}, and particularly $d$-degenerate graphs~\cite{CvdHJ1}.  Now
much emphasis has shifted to bounding the diameter of these recoloring
graphs~\cite{cereceda-diss,BH,BP,feghali-short}.

The present paper is motivated by work of Dvo\v{r}\'{a}k and Feghali~\cite{DF1,DF2}.  
They showed that if an $n$-vertex graph $G$ is planar, then any 10-coloring can be
transformed to any other and, more strongly, that the corresponding recoloring
graph has diameter at most $8n$.  They proved an analagous result for 7-colorings
of planar triangle-free graphs.  Further, they conjectured that the same result
holds more generally for $L$-colorings from an arbitrary list-assignment $L$
where, for all $v$, we have $|L(v)|\ge 7$.
Theorem~\ref{thm1A} confirms this conjecture.  They also conjectured an
analagous result for 6-colorings of planar graphs with girth at least 5.
Theorem~\ref{thm2A} confirms this latter conjecture.  Bonamy et al.~\cite{BJLPP}
showed that the 3-recoloring graph of the $n$-vertex path $P_n$ has diameter
$\Theta(n^2)$.  Thus, if we are trying to prove linear upper bounds on the
distance between any two $L$-colorings of $G$, we typically consider
list-assignments $L$ where always $|L(v)|\ge 4$.  Thus, Theorem~\ref{thm3A} is the
natural continuation of this line of study to sparser graphs.

\section{Preliminaries}
\label{prelims-sec}

An \Emph{$L$-recoloring sequence} consists of an $L$-coloring and a sequence
of $L$-recoloring steps, so that each later coloring is a proper $L$-coloring.
An $L$-recoloring sequence is \EmphE{$k$-good}{4mm} if each vertex is
recolored at most $k$ times.
An $L$-coloring $\alpha$ of $G$ restricted to a subgraph $G'$ is denoted
\EmphE{$\alpha_{\uphr G'}$}{4mm}.
We will often \EmphE{extend}{4mm} an $L$-recoloring sequence $\sigma'$ that
transforms $\alpha_{\uphr G'}$ to $\beta_{\uphr G'}$ to an $L$-recoloring
sequence $\sigma$ that transforms $\alpha$ to $\beta$.  This means that 
restricting $\sigma$ to $V(G')$ results in $\sigma'$.

The following Key Lemma is simple but powerful.  It has been used implicitly in many
papers, and first appeared explicitly in~\cite{BBFHMP}.  We phrase it in the
slightly more general language of list-coloring, although the proof is
identical.  In fact, the reader familiar with correspondence coloring will note
that the proof works equally well in that, still more general, context, too.
The same is true of all the proofs in this paper.

\begin{keylem}
Fix a graph $G$, a list-assignment $L$ for $G$, and $L$-colorings $\alpha$
and $\beta$.  Fix a vertex $v$ 
with $|L(v)|\ge d(v)+1$ 
and let $G':=G-v$.  Fix an $L$-recoloring
sequence $\sigma'$ for $G'$ transforming $\alpha_{\uphr G'}$ to $\beta_{\uphr
G'}$.  If $\sigma'$ recolors $N_G(v)$  a total of $t$ times, then we can extend
$\sigma'$ to an $L$-recoloring sequence for $G$ transforming $\alpha$ into
$\beta$ and recoloring $v$ at most $\lceil t/(|L(v)|-d_G(v)-1)\rceil+1$ times.
\end{keylem}

\begin{proof}
Let $c_1,c_2,\ldots$ denote the sequence of colors used by $\sigma'$ to recolor
vertices in $N(v)$, in order and possibly with repetition.  Let $s:=|L(v)|-d_G(v)-1$.
Immediately before $\sigma'$ is to recolor a vertex in $N(v)$ with $c_1$, we recolor $v$
with a color distinct from $c_1,\ldots,c_s$ and also
distinct from all colors currently used on $N(v)$.  This is possible
precisely because $|L(v)|=s+d_G(v)+1$. Immediately before $\sigma$ is to recolor a vertex in 
$N(v)$ with $c_{s+1}$, we recolor $v$ with a color distinct from
$c_{s+1},\ldots,c_{2s}$ (and those currently used on $N(v)$).  Continuing in
this way, we extend $\sigma'$ to an $L$-recoloring
sequence that recolors $v$ at most $\lceil t/s\rceil$ times and transforms
$\alpha$ to a coloring $\beta'$ that agrees with $\beta$ everywhere except
possibly on $v$.  Finally, if needed, we recolor $v$ with $\beta(v)$.
So $v$ is recolored at most $\lceil t/s\rceil+1$ times.
\end{proof}

For easy reference we include the following lemma, which is folklore.  

\begin{mad-lem}
\label{mad:lem}
If $G$ is a planar graph with girth at least $g$, then $\mad(G)<2g/(g-2)$.
\end{mad-lem}

For completeness we conclude this short section with some standard
definitions.  We denote the degree of a vertex $v$ by $d(v)$.  For a planar
graph $G$ and a face $f$ in a plane embedding of $G$, we denote the length of
$f$ by $\ell(f)$.
A $k$-vertex is a vertex of degree $k$.  A $k^+$-vertex
(resp.~$k^-$-vertex) is one of degree at least (resp.~at most)
$k$.\aaside{$k$/$k^+$/$k^-$-vertex}{-8mm}
A \EmphE{$k$/$k^+$/$k^-$-neighbor}{-4mm} of a given vertex is an adjacent
$k/k^+/k^-$-vertex.
A \EmphE{list-assignment}{0mm} $L$ for a graph $G$ gives each vertex $v$
a list $L(v)$ of allowable colors.  If $|L(v)|=k$ for all $v$, then $L$ is a
\Emph{$k$-assignment}.  An \EmphE{$L$-coloring}{4mm} is a proper coloring $\alpha$
such that $\alpha(v)\in L(v)$ for all $v$.  The \emph{maximum average degree} of
a graph $G$ is the maximum, over all nonempty subgraphs $H$ of $G$, of the
average degree of $H$; it is denoted \Emph{$\mad(G)$}.  That is,
$\mad(G):=\max_{H\subseteq G}2|E(H)|/|V(H)|$.

\section{Triangle-free Planar Graphs: Lists of Size 7}
\label{7list-sec}
In this section we prove Theorem~\ref{thm1A}.  To simplify the proof, it is
convenient to extract the following structural lemma about triangle-free planar
graphs.

\begin{lem}
\label{girth4-lem}
If $G$ is triangle-free and planar and $\delta(G)\ge 3$, then $G$ contains
one of the following four configurations: 
\begin{enumerate}
\item[(a)] a 5-vertex with at least three 3-neighbors, 
\item[(b)] a path with each endpoint of
degree 3 and at most three internal vertices, all of degree 4, 
\item[(c)] a 4-face with an incident 3-vertex and three incident 4-vertices, or 
\item[(d)] a 4-face $f$ with degree sequence (3,4,5,4) or (3,4,4,5) such that
the 5-vertex on $f$ is also adjacent to a 3-vertex not on $f$.
\end{enumerate}
\end{lem}
\begin{proof}
Assume the lemma is false, and let $G$ be a counterexample.
We use discharging, giving each vertex $v$ initial charge $2d(v)-6$ and each
face $f$ initial charge $\ell(f)-6$.  By Euler's formula, the sum of all
initial charges is $-12$.  We use the following three discharging rules.
\begin{enumerate}
\item [(R1)] Each 4-vertex gives 1/2 to each incident face.\footnote{In this rule
and those following, if a vertex $v$ is a cut-vertex and is incident to a face
$f$ at multiple points on a walk along the boundary of $f$, then $v$ gives $f$
charge for each incidence.}
\item [(R2)] Each $6^+$-vertex gives 1 to each incident face.
\item [(R3)] After receiving charge from (R1) and (R2), if any face $f$ still
needs charge, then it takes that charge equally from all incident 5-vertices.
\end{enumerate}
Now we show that each vertex and face ends with nonnegative charge.  This
contradicts that the sum of the initial charges is $-12$, and proves the lemma.
We write $\ch^*(v)$ to denote the final charge of $v$, that is, the charge of $v$
after applying each of (R1)--(R3).

By assumption, $G$ has no $2^-$-vertices.  Each 3-vertex starts and ends with
charge 0.  Each 4-vertex $v$ has $\ch^*(v)= 2(4)-6-4(1/2)=0$.  And each
$6^+$-vertex $v$ has $\ch^*(v)\ge 2d(v)-6-d(v)\ge 0$, since $d(v)\ge 6$.
Clearly each $6^+$-face finishes with charge nonnegative.
By (R1), (R2), and (b), the same is true for each 5-face; and
by (R3), the same is true for each face with an incident 5-vertex.
So assume that $f$ is a 4-face with no incident 5-vertex.  By (R1), (R2), and
(b), it is easy to check that $f$ finishes with charge nonnegative; here we use
that (c) is forbidden.

Thus, what remains is to show that each 5-vertex finishes with charge
nonnegative.  First, we classify the faces (and degrees of their incident
vertices) that take more than $1/2$ from an incident 5-vertex.
Each $6^+$-face $f$ finishes with at least its initial charge $\ell(f)-6\ge 0$.
By (b), $G$ contains no adjacent 3-vertices, so each 5-face has at least 3
incident $4^+$-vertices, and thus takes at most $1/3$ from each incident
5-vertex.  So consider a 4-face $f$ with an incident 5-vertex.  If $f$ has at
least two incident $6^+$-vertices, then it receives at least 2 by (R2), so
takes no charge from incident 5-vertices.  If $f$ has exactly one incident
$6^+$-vertex and at most one incident 3-vertex, then $f$ takes at most $1/2$
from each incident 5-vertex.  Similarly, if $f$ has no incident 3-vertices,
then $f$ takes at most $1/2$ from each incident 5-vertex.  So assume that $f$
either has two incident 3-vertices or else has one incident 3-vertex and no
incident $6^+$-vertices.  By (b), if $f$ has two incident
3-vertices, then each other incident vertex is a $5^+$-vertex.

\begin{figure}[!ht]
\centering
\begin{tikzpicture}[thick, scale=.4725]
\tikzstyle{uStyle}=[shape = circle, minimum size = 4.5pt, inner sep = 0pt,
outer sep = 0pt, draw, fill=white, semithick]
\tikzstyle{lStyle}=[shape = circle, minimum size = 14.5pt, inner sep = 0pt,
outer sep = 0pt, draw=none, fill=none]
\tikzset{every node/.style=uStyle}
\def\off{.4}

\draw (0,0) node (A) {} -- (0,2) node (B) {} -- (2,2) node (C) {} -- (2,0) node
(D) {} -- (A);
\draw (A) ++ (-\off,-\off) node[lStyle] {\footnotesize{5}};
\draw (B) ++ (-\off,\off) node[lStyle] {\footnotesize{3}};
\draw (C) ++ (1.5*\off,1.2*\off) node[lStyle] {\footnotesize{$5^+$}};
\draw (D) ++ (\off,-\off) node[lStyle] {\footnotesize{3}};
\draw (1,1) node[lStyle] (E) {\footnotesize{1}};
\draw (A) ++ (.5*\off,.5*\off) edge [->] (E);

\begin{scope}[xshift=1.8in]
\draw (0,0) node (A) {} -- (0,2) node (B) {} -- (2,2) node (C) {} -- (2,0) node
(D) {} -- (A);
\draw (A) ++ (-\off,-\off) node[lStyle] {\footnotesize{5}};
\draw (B) ++ (-\off,\off) node[lStyle] {\footnotesize{4}};
\draw (C) ++ (\off,\off) node[lStyle] {\footnotesize{3}};
\draw (D) ++ (\off,-\off) node[lStyle] {\footnotesize{4}};
\draw (1,1) node[lStyle] (E) {\footnotesize{1}};
\draw (A) ++ (.5*\off,.5*\off) edge [->] (E);
\end{scope}

\begin{scope}[xshift=3.6in]
\draw (0,0) node (A) {} -- (0,2) node (B) {} -- (2,2) node (C) {} -- (2,0) node
(D) {} -- (A);
\draw (A) ++ (-\off,-\off) node[lStyle] {\footnotesize{5}};
\draw (B) ++ (-\off,\off) node[lStyle] {\footnotesize{4}};
\draw (C) ++ (\off,\off) node[lStyle] {\footnotesize{4}};
\draw (D) ++ (\off,-\off) node[lStyle] {\footnotesize{3}};
\draw (1,1) node[lStyle] (E) {\footnotesize{1}};
\draw (A) ++ (.5*\off,.5*\off) edge [->] (E);
\end{scope}

\begin{scope}[xshift=5.4in]
\draw (0,0) node (A) {} -- (0,2) node (B) {} -- (2,2) node (C) {} -- (2,0) node
(D) {} -- (A);
\draw (A) ++ (-\off,-\off) node[lStyle] {\footnotesize{5}};
\draw (B) ++ (-\off,\off) node[lStyle] {\footnotesize{4}};
\draw (C) ++ (\off,\off) node[lStyle] {\footnotesize{5}};
\draw (D) ++ (\off,-\off) node[lStyle] {\footnotesize{3}};
\draw (1,1) node[lStyle] (E) {\footnotesize{3/4}};
\draw (A) ++ (.5*\off,.5*\off) edge [->] (E);
\end{scope}

\begin{scope}[xshift=7.2in]
\draw (0,0) node (A) {} -- (0,2) node (B) {} -- (2,2) node (C) {} -- (2,0) node
(D) {} -- (A);
\draw (A) ++ (-\off,-\off) node[lStyle] {\footnotesize{5}};
\draw (B) ++ (-\off,\off) node[lStyle] {\footnotesize{5}};
\draw (C) ++ (\off,\off) node[lStyle] {\footnotesize{4}};
\draw (D) ++ (\off,-\off) node[lStyle] {\footnotesize{3}};
\draw (1,1) node[lStyle] (E) {\footnotesize{3/4}};
\draw (A) ++ (.5*\off,.5*\off) edge [->] (E);
\end{scope}

\begin{scope}[xshift=9.0in]
\draw (0,0) node (A) {} -- (0,2) node (B) {} -- (2,2) node (C) {} -- (2,0) node
(D) {} -- (A);
\draw (A) ++ (-\off,-\off) node[lStyle] {\footnotesize{5}};
\draw (B) ++ (-\off,\off) node[lStyle] {\footnotesize{5}};
\draw (C) ++ (\off,\off) node[lStyle] {\footnotesize{5}};
\draw (D) ++ (\off,-\off) node[lStyle] {\footnotesize{3}};
\draw (1,1) node[lStyle] (E) {\footnotesize{2/3}};
\draw (A) ++ (.5*\off,.5*\off) edge [->] (E);
\end{scope}

\end{tikzpicture}
\caption{The six types of 4-faces, up to rotation and reflection, that take
charge more than 1/2 from each
incident 5-vertex $v$. Type 1 is on the left, and the types increase from left to
right.  Along with each type, we indicate in the center of the 4-face $f$ the
amount of charge that $f$ takes from $v$.\label{needy4faces-fig} 
(Each integer near a vertex denotes the degree of that vertex.)
}
\end{figure}
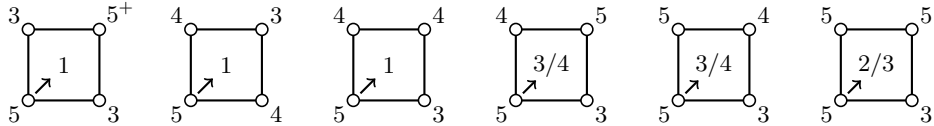

\aaside{types 1--6}{-2cm}
It is easy to check that if $f$ takes more than $1/2$ from an
incident 5-vertex, then $f$ has one of the six types shown in
Figure~\ref{needy4faces-fig}.
Further, if a 5-vertex $v$ is incident to at most one face of types 1--3, then
each other face incident to $v$ takes at most 3/4 from $v$,
so $\ch^*(v) \ge 2(5)-6-1-4(3/4)=0$.  Thus, below we assume that $v$
is incident to at least two faces of types 1--3.

Suppose that a 5-vertex $v$ is incident to a 4-face of type 1.  By (d), $v$ is
not incident to any 4-face of type 2 or type 3.  And by (a), $v$ is not incident
to another 4-face of type 1.  Thus, $v$ is incident with at most one face of
type 1--3, a contradiction.

Suppose that a 5-vertex $v$ is incident to a 4-face of type 2. By (d), $v$ is
not incident to any 4-face of type 3.  
So assume that $v$ is incident to (exactly) two 4-faces of type 2.  
By (b), these faces do not share a 4-vertex.
Let $f_1,\ldots,f_5$ denote the 5 faces incident to $v$ (with
multiplicity, if $v$ is a cut-vertex) in cyclic order.  By symmetry, we assume
that $f_1$ and $f_3$ are type 2.  However, now $f_2$ receives at most $1/2$ from
$v$.  Thus, $\ch^*(v)\ge 2(5)-6-2(1)-1/2-2(3/4)=0$.

Suppose that a 5-vertex $v$ is incident to a 4-face of type 3.  
In fact, $v$ must be incident to two such faces.
By (d), these type 3 faces must share a common 3-vertex.  So, by symmetry, we
assume that $f_2$ and $f_3$ are type 3, and they share a 3-vertex.  By (d), $v$
has no other 3-neighbor.  And by (b), each of $f_1$ and $f_4$ is either a
$5^+$-face or a 4-face with no incident 3-vertex.  Thus, $f_1$ and $f_4$
each take at most $1/2$ from $v$.  So $\ch^*(v)\ge 2(5)-6-2(1)-3/4-2(1/2)>0$.
\end{proof}

\begin{thm}
\label{thm1}
Let $G$ be a triangle-free planar graph, and let $L$ be a 7-assignment for $G$.
If $\alpha$ and $\beta$ are $L$-colorings of $G$, then $\alpha$ can be
transformed to $\beta$ by recoloring each vertex at most 30 times, so
that every intermediate coloring is a proper $L$-coloring.
\end{thm}

\begin{proof}
Suppose the theorem is false, and let $G$ be a smallest counterexample.  We
show that $\delta(G)\ge 3$, so Lemma~\ref{girth4-lem} applies.  We then show
that each of configurations (a)--(d) in Lemma~\ref{girth4-lem} is reducible for
the present theorem.  Thus, no counterexample exists, and the theorem is true.

Recall that an $L$-recoloring sequence is \emph{30-good} if it recolors each vertex at most
30 times.  If $G$ contains a $2^-$-vertex $v$, then by minimality $G-v$ has a
30-good recoloring sequence $\sigma'$ to transforms $\alpha_{\uphr G-v}$ to
$\beta_{\uphr G-v}$.  By the Key Lemma, we extend $\sigma'$ to a 30-good
recoloring sequence for $G$ that recolors $v$ at most $(30+30)/4+1=16$ times.
So $\delta(G)\ge 3$, and Lemma~\ref{girth4-lem} applies.

(a) Suppose $G$ contains a 5-vertex $v$ with neighbors $w_1,\ldots,w_5$
such that $d(w_1)=d(w_2)=d(w_3)=3$.  Let $G':=G-\{w_1,w_2,w_3\}$ and
$G'':=G'-v$.  By minimality, $G''$ has a 30-good recoloring sequence
$\sigma''$ transforming $\alpha_{\uphr G''}$ to $\beta_{\uphr G''}$.
By the Key Lemma, we can extend $\sigma''$ to a 30-good recoloring
sequence $\sigma'$ for $G'$ transforming $\alpha_{\uphr G'}$ to $\beta_{\uphr
G'}$ that recolors $v$ at most $(30+30)/4+1=16$ times.  By applying the Key Lemma
to each of $w_1$, $w_2$, $w_3$ (in succession), we can extend $\sigma'$ to a
30-good recoloring sequence $\sigma$ for $G$ that recolors each of $w_1$,
$w_2$, $w_3$ at most $\lceil (30+30+16)/3\rceil+1=27$ times.

(b) Suppose $G$ contains a path $v_1v_2v_3v_4v_5$ with $d(v_1)=d(v_5)=3$
and $d(v_2)=d(v_3)=d(v_4)=4$.  Let $G':=G-\{v_1,v_5\}$, let
$G'':=G'-\{v_2,v_4\}$, and let $G''':=G''-\{v_3\}$.  By minimality, $G'''$ has a
30-good $L$-recoloring sequence $\sigma'''$ transforming $\alpha_{\uphr
G'''}$ to $\beta_{\uphr G'''}$.  By the Key Lemma, we extend $\sigma'''$ to
a 30-good $L$-recoloring sequence $\sigma''$ for $G''$ transforming
$\alpha_{\uphr G''}$ to $\beta_{\uphr G''}$ that recolors $v_3$ at most
$(30+30)/4+1=16$ times.  By using the Key Lemma twice, we extend $\sigma''$
to a 30-good $L$-recoloring sequence $\sigma'$ for $G'$ transforming
$\alpha_{\uphr G'}$ to $\beta_{\uphr G'}$ that recolors each of $v_2$ and $v_4$
at most $\lceil(30+30+16)/3\rceil+1=27$ times.  Finally, by using the Key Lemma
twice more, we extend $\sigma'$ to a 30-good $L$-recoloring sequence $\sigma$
for $G$ transforming $\alpha$ to $\beta$; note that we recolor each of $v_1$
and $v_5$ at most $(30+30+27)/3+1=30$ times.
When the path has length $t$, for a $t\le 3$, the analysis is similar, but
ends after only extending $\sigma'''$ to $t+1$ vertices.

(c) Suppose $G$ contains a 4-face $f$ with degree sequence (3,4,4,4).
The proof is nearly identical to that of (b).  
Only now the
3-vertex is recolored at most $(30+27+27)/3+1=29$ times.

(d) Suppose $G$ contains a 4-face $f$ with degree sequence (3,4,5,4) or
(3,4,4,5) and the 5-vertex $v$ on $f$ is also adjacent to a 3-vertex $w$ not on $f$.
Form $G'$ from $G$ by deleting $w$, and form $G''$
from $G'$ by deleting the vertices of $f$.  As in (c) and (b), $G''$ has
a 30-good $L$-recoloring sequence, and we can extend it a 30-good $L$-recoloring
sequence for $G'$ that transforms $\alpha_{\uphr G'}$ to $\beta_{\uphr G'}$ so
that each vertex other than the 3-vertex is recolored at most 27 times.  
(We extend to the vertices of $f$ in order of decreasing distance from its
3-vertex.)
By the
Key Lemma we can extend $\sigma'$ to a 30-good $L$-recoloring sequence for $G$
that transforms $\alpha$ to $\beta$.  This is because one neighbor of $w$
is recolored by $\sigma'$ at most 27 times.
\end{proof}

\section{Graphs with Mad \texorpdfstring{$\bm{<17/5}$}{17/5}: Lists of Size 6}
\label{6list-sec}

The goal of this section is to prove Theorem~\ref{thm2}.  Very generally, the
proof is similar to that of Theorem~\ref{thm1}.  However, the Key Lemma is not
as well suited to working with lists of even size.  So, rather than applying it
directly, we will typically prove more specialized results that give better
bounds in our particular cases.  However, the proofs will be quite similar to
that of the Key Lemma.  This remark also applies to Section~\ref{4list-sec},
since it considers lists of size 4.

\begin{thm}
\label{thm2}
Let $G$ be a graph with $\mad(G)<17/5$. If $L$ is a 6-assignment for $G$ and
$\alpha$ and $\beta$ are $L$-colorings of $G$, then $\alpha$ be can be
transformed to $\beta$ by recoloring each vertex at most 12 times, so
that every intermediate coloring is a proper $L$-coloring.  
In particular, this holds for every planar graph of girth at least 5.
\end{thm}

\begin{proof}
By the 
\hyperref[mad:lem]{Mad Lemma}, if $G$ is planar with girth at least 5, then
$\mad(G)<10/3<17/5$.  So, the second statement follows from the first.  Now we
prove the first.

Let $G$ be a graph with $\mad(G)<17/5$.  We claim that $G$ contains either (i)
a $2^-$-vertex or (ii) a 3-vertex with at least two 3-neighbors or (iii) a
4-vertex with four 3-neighbors.  We assume that $G$ is a counterexample to the
claim, and use discharging to reach a contradiction.  
We give each vertex charge equal to its degree and use a single
discharging rule: Each 3-vertex takes 1/5 from each $4^+$-neighbor.  By
assumption, $\delta(G)\ge 3$.  If $v$ is a 3-vertex, then the absence of (ii)
implies that $v$ has at least two $4^+$-neighbors. So $\ch^*(v) \ge 3+2(1/5) =
17/5$.  If $v$ is a 4-vertex, then the absence of (iii) implies that
$v$ has at most three 3-neighbors.  So $\ch^*(v)\ge 4-3(1/5) = 17/5$.  If $v$
is a $5^+$-vertex, then $\ch^*(v)\ge d(v)-d(v)/5 = 4d(v)/5 \ge 4$, since
$d(v)\ge 5$.  Since $\ch^*(v)\ge 17/5$ for all $v$, we contradict
our hypothesis $\mad(G)<17/5$.  This proves the claim.

We show that none of (i), (ii), and (iii) appears in a minimal
counterexample to the theorem.  

(i) Suppose $G$ has a $2^-$-vertex $v$.  
Let $G':=G-v$.  By minimality,
$G'$ has a 12-good $L$-recoloring sequence $\sigma'$ transforming
$\alpha_{\uphr G'}$ to $\beta_{\uphr G'}$.  By the Key Lemma, we extend
$\sigma'$ to a 12-good $L$-recoloring sequence transforming $\alpha$ to
$\beta$, recoloring $v$ at most $(12+12)/3+1=9$ times.

(ii) Suppose $G$ has a 3-vertex $v$ with at least two 3-neighbors, $w_1$
and $w_2$.\aside{$v$, $w_1$, $w_2$}  Let $G':=G-\{v,w_1,w_2\}$.\aside{$G'$,
$\sigma'$}  
By minimality, $G'$ has a 12-good $L$-recoloring sequence $\sigma'$ that
transforms $\alpha_{\uphr G'}$ to $\beta_{\uphr G'}$.  Let $a_1,a_2,\ldots$
denote the colors that are used by $\sigma'$ to recolor $N(w_1)\setminus
\{v\}$, let $b_1,b_2,\ldots$ denote the colors that are used by $\sigma'$
to recolor $N(w_2)\setminus \{v\}$, and let $c_1,c_2,\ldots$ denote the colors that
are used by $\sigma'$ to recolor $N(v)\setminus \{w_1,w_2\}$.  
(Each of $a_1,\ldots$ and $b_1,\ldots$ and $c_1,\ldots$ is in order and with
repetition.) The proof is similar to that of the Key Lemma, but a bit more
involved.  For simplicity, we assume that $w_1w_2\notin E(G)$, since the other
case is similar but easier.

When $\sigma'$ is about to recolor $N(w_1)\setminus \{v\}$ with $a_1$, we
first recolor $w_1$ to avoid $\{a_1,a_2,a_3\}$, as well as the two colors
currently used on $N(w_1)\setminus\{v\}$.  If this requires coloring $w_1$ with
the color currently on $v$, then before recoloring $w_1$ we recolor $v$ (as we
will explain shortly).  When $\sigma'$ is about to recolor $N(w_1)\setminus \{v\}$
with $a_4$, we first recolor $w_1$ to avoid $\{a_4,a_5,a_6\}$, again recoloring
$v$ beforehand if needed.  When $\sigma'$ is about to recolor $N(w_1)\setminus
\{v\}$ with $a_7$, we first recolor $w_1$ to avoid $\{a_7,a_8\}$ and all colors
currently used on $N(w_1)$.  Thereafter, each time that $\sigma'$ specifies
recoloring $N(w_1)\setminus \{v\}$ with the color currently on $w_1$, we first
recolor $w_1$ to avoid the next two colors to be used by $\sigma'$ on
$N(w_1)\setminus \{v\}$, as well as the three colors currently used on $N(w_1)$.
Finally, after $N(w_1)\setminus \{v\}$ is recolored to agree with $\beta$, we
recolor $w_1$ with $\beta(w_1)$, again recoloring $v$ beforehand if needed.
Since $\sigma'$ is 12-good, at most 24 colors are used by $\sigma'$ on
$N(w_1)\setminus\{v\}$.  So the number of times that $w_1$ is recolored is at
most $2+(24-6)/2+1=12$.  We treat $w_2$ in exactly the same way as $w_1$.  So
what remains is to specify how we handle $v$.

Denote $N(v)\setminus\{w_1,w_2\}$ by $\{x\}$.\aside{$x$}
Each time that $\sigma'$ is about to recolor $x$ with a color currently used on
$v$, we first recolor $v$ to avoid the next two colors used by $\sigma'$ on $x$,
as well as the three colors currently used on $N(v)$.  Each time that some $w_i$ 
requires $v$ to be recolored (at most twice for each $w_i$ near the start,
and at most once total near the end, for up to five times total, as we explain
below), before $w_i$ is recolored, we recolor $v$ to avoid its current
color, the up to three colors currently used on $N(v)$, as well as the next
color that will be used on $x$, i.e., the next $c_j$.  

Finally, assume
that $G$ has been recolored from $\alpha$ to some $\beta'$ that agrees with
$\beta$ except possibly on $\{v,w_1,w_2\}$.  
Recolor $v$ to avoid $\{\beta'(w_1), \beta'(w_2), \beta(w_1), \beta(w_2),
\beta(x)\}$, recolor each $w_i$ to $\beta(w_i)$, then recolor $v$ to $\beta(v)$.
So how many times is $v$
recolored?  It might be recolored due to each $w_i$ at most twice early on.
It might be recolored $12/2$ times due to colors used by $\sigma'$ on $x$.
It is subtle, but true, that recoloring $v$ due to some $w_i$ does not cause us
to ``lose ground'' with respect to upcoming colors to be used on $x$.  Each time
that we recolor $v$ due to $x$, we do so to avoid the next two colors to be used
on $x$.  However, immediately after we recolor $v$, we use the first of those
colors on $x$.  So the color of $v$ is chosen to avoid at most one upcoming
color on $x$.  Whenever $v$ is recolored due to some $w_i$, the color for $v$ is
again chosen to avoid the next upcoming color used on $x$.  So the total number
of times that $v$ is recolored is $2+2+12/2+2=12$. 

(iii) Suppose $G$ contains a 4-vertex $v$ with four 3-neighbors,
$w_1,w_2,w_3,w_4$.  (Again, for simplicity we assume that $w_iw_j\notin E(G)$
for all $i,j\in\{1,2,3,4\}$; the other cases are similar.)
The proof is nearly identical to that for (ii) above, with each
$w_i$ being treated as $w_1$ and $w_2$ were above.  Suppose that $\alpha$ is
transformed to some $\beta'$ that agrees with $\beta$ except on
$\{v,w_1,w_2,w_3,w_4\}$.  Now each $w_i$ such that $\beta(w_i)\ne \beta'(v)$ is
recolored to $\beta(w_i)$.  Next $v$ is recolored to avoid its current color and
those currently on $N(v)$.  Afterward, each remaining $w_i$ is recolored to
$\beta(w_i)$.  Finally, $v$ is recolored to $\beta(v)$.  The analysis for each
$w_i$ is identical to that for (ii).  And $v$ is recolored at most $2(4)+2=10$
times.
%
\end{proof}

For planar graphs of girth at least 5, we can improve the above result to a
10-good $L$-recoloring sequence.  This is because such graphs contain either a
$2^-$-vertex or a path of length 3 with all vertices of degree 3.  (In fact, the
vertices of such a path lie on a 5-face $f$, and the fifth vertex on $f$ is a
$5^-$-vertex.)  This was proved in~\cite[Lemma~1]{CY-linear}.  To show that this
configuration is reducible (for a 10-good $L$-recoloring sequence), the analysis 
is similar to that for (ii) above.  However, we prefer the
proof given because it (a) is self-contained, (b) highlights the fact that
planarity is not needed, only sparseness, and (c) works for the larger class of
graphs with $\mad<17/5$ (not just those with $\mad<10/3$).  In fact, with more
work the bound on $\mad$ can be increased a bit.  But we have opted for a
simpler proof of a slightly weaker result.

\section{Graphs with Mad \texorpdfstring{$\bm{<22/9}$}{22/9}: Lists of Size 4}
\label{4list-sec}

The goal of this section is to prove the following result.

\begin{thm}
\label{thm3}
Let $G$ be a graph with $\mad(G)<22/9$. If $L$ is a 4-assignment for $G$ and
$\alpha$ and $\beta$ are $L$-colorings of $G$, then $\alpha$ be can be
transformed to $\beta$ by recoloring each vertex at most 14 times, so
that every intermediate coloring is a proper $L$-coloring.  In particular, this
holds for every planar graph with girth at least 11.
\end{thm}

Note that the latter statement follows from the former by the 
\hyperref[mad:lem]{Mad Lemma}; so we
prove the former.  We assume $G$ is a smallest counterexample,
and again use discharging.  This time we show that $G$ has average
degree at least $22/9$, contradicting the hypothesis.  Generally
speaking, we send charge from $3^+$-vertices, which begin with extra
charge, to 2-vertices, which begin needing more charge.  To prove that each
$3$-vertex ends with sufficient charge, we show that various
configurations (see Section~\ref{4list-reducibility-sec}) cannot appear in
our minimum counterexample $G$.

A \emph{thread} in $G$ is a path\footnote{We also allow a thread to be a cycle
with a single $3^+$-vertex (which serves as both endpoints of the thread) and
all other vertices of degree 2.} with all internal vertices of degree 2.
A \Emph{$k$-thread} is a thread with $k$ internal vertices.  
A \Emph{$3_{a,b,c}$-vertex} is a 3-vertex that is the
endpoint of a maximal $a$-thread, a maximal $b$-thread, and a maximal
$c$-thread, all distinct.
A \Emph{weak neighbor} of a $3^+$-vertex $v$ is a $3^+$-vertex $w$ such that $v$
and $w$ are endpoints of a common thread.  A 2-vertex $v$ is \Emph{nearby} a
$3^+$-vertex $w$ if $v$ is an interior vertex of a thread with $w$ as one endpoint.

Since the proof is longer than our previous proofs, we have one subsection for
reducibility and another for discharging.  In both subsections, $G$ is a
minimum counterexample to Theorem~\ref{thm3}.

\subsection{Reducibility}
\label{4list-reducibility-sec}
\begin{lem}
\label{lem0}
$G$ is connected and $\delta(G)\ge 2$.
\end{lem}
\begin{proof}
If $G$ is disconnected, then each component has a 14-good $L$-recoloring sequence,
by minimality.  Combining these gives a 14-good $L$-recoloring sequence for $G$.
So $G$ is connected.  Suppose instead that $G$ contains a 1-vertex $v$ and let $G':=G-v$. 
By minimality, $G'$ has a 14-good $L$-recoloring sequence $\sigma'$ that
transforms $\alpha_{\uphr G'}$ to $\beta_{\uphr G'}$.  Since the neighborhood
of $v$ is recolored at most 14 times, by the Key Lemma we can extend $\sigma'$
to a 14-good recoloring sequence for $G$ that recolors $v$ at most $14/2+1=8$ times.  
\end{proof}

\begin{lem}
\label{lem1}
Let $v_1v_2v_3v_4$ be a 2-thread in some subgraph $H$ of $G$.  Let
$H':=H-\{v_2,v_3\}$.  Let $\sigma'$ be a 14-good $L$-recoloring sequence for $H'$
that transforms $\alpha_{\uphr H'}$ to $\beta_{\uphr H'}$.  If $\sigma'$
recolors $v_4$ at most $s$ times, and $s\le 11$, then $H$ has a 14-good
$L$-recoloring sequence that recolors $v_3$ at most $s+3$ times and transforms
$\alpha$ to $\beta$.
\end{lem}

\begin{proof}
Let $a_1,\ldots$
denote the sequence of new colors used by $\sigma'$ on $v_1$ and let
$b_1,\ldots$ denote the sequence of new colors used by $\sigma'$ on $v_4$.  For
convenience, let $a_0:=\alpha(v_1)$ and $b_0:=\alpha(v_4)$.
Each time that $\sigma'$ recolors $v_4$ from $b_i$ to
$b_{i+1}$, we first, if needed, recolor $v_3$ to avoid $b_i, b_{i+1}$, and the
color currently used on $v_2$.  Since $L$ is a 4-assignment, $v_3$ always has an
available color.  When $\sigma'$ recolors $v_1$ from  $a_0$ to $a_1$, we first
recolor $v_2$ to avoid $a_0,a_1,a_2$; if this requires recoloring $v_2$ with
the color currently used on $v_3$, then beforehand we recolor $v_3$ to avoid
the colors currently on $v_2, v_3, v_4$.  Afterward, we recolor $v_2$ to avoid
$a_0,a_1,a_2$.  
Thereafter, each time that $\sigma'$ recolors $v_1$ from $a_i$ to
$a_{i+1}$, we first recolor $v_2$ to avoid $a_i$, $a_{i+1}$, and the color
currently used on $v_3$.  

This process recolors $\alpha$ to an $L$-coloring $\beta'$
that agrees with $\beta$ everywhere except possibly on $v_2$ and $v_3$.
Now, if needed, we recolor $v_2$ with $\beta(v_2)$.  If
$\beta(v_2)=\beta'(v_3)$, then we first recolor $v_3$ to some
other arbitrary color, currently unused on $v_2$, $v_3$, and $v_4$.  Finally,
if needed, we recolor $v_3$ with $\beta(v_3)$.
The number of recolorings of $v_2$ is at most $1+(14-2)+1=14$.
The number of recolorings of $v_3$ is at most $1+s+2\le 14$.
\end{proof}

\begin{lem}
\label{lem2}
Let $v_1v_2v_3v_4v_5$ be a 3-thread in some subgraph $H$ of $G$ (not necessarily
proper).  
Let $H':=H-\{v_2,v_3,v_4\}$.  Let $\sigma'$ be a 14-good
$L$-recoloring sequence for $H'$ that transforms $\alpha_{\uphr H'}$
to $\beta_{\uphr H'}$.
We can extend $\sigma'$ to a 14-good $L$-recoloring sequence for $H$
that transforms $\alpha_{\uphr H}$ to $\beta_{\uphr H}$ and recolors $v_3$
at most 4 times.  In particular, $G$ has no 3-threads.
\end{lem}

\begin{proof}
See the left of Figure~\ref{lem2,3-fig}.
The proof is very similar to that of Lemma~\ref{lem1}, so we just sketch the
details.  We essentially treat both $v_2$ and $v_4$ in the way that we treated
$v_2$ in the proof of Lemma~\ref{lem1}; we extend the recoloring sequence to
$v_2$ and $v_4$ so that in total they require $v_3$ to be recolored at most
three times (once each near the start and once together near the end to recolor
$v_2$ and $v_4$ with $\beta(v_2)$ and $\beta(v_4)$).  Finally, we may need
to recolor $v_3$ at the end, so that it uses the color $\beta(v_3)$.  
Thus, $v_3$ is recolored at most $2(1)+1+1=4$ times.
\end{proof}

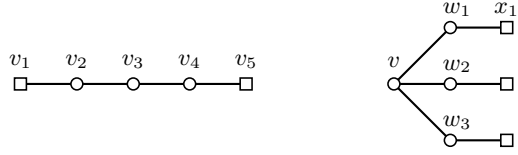
\begin{figure}[!h]
\centering
\begin{tikzpicture}[thick, scale=.75]
\tikzstyle{uStyle}=[shape = circle, minimum size = 4.5pt, inner sep = 0pt,
outer sep = 0pt, draw, fill=white, semithick]
\tikzstyle{sStyle}=[shape = rectangle, minimum size = 4.5pt, inner sep = 0pt,
outer sep = 0pt, draw, fill=white, semithick]
\tikzstyle{lStyle}=[shape = circle, minimum size = 4.5pt, inner sep = 0pt,
outer sep = 0pt, draw=none, fill=none]
\tikzset{every node/.style=uStyle}
\def\off{.4}

\draw (1,0) node[sStyle] (v1) {} --
(2,0) node (v2) {} --
(3,0) node (v3) {} --
(4,0) node (v4) {} --
(5,0) node[sStyle] (v5) {};
\foreach \i in {1,...,5}
\draw (v\i) ++ (0,\off) node[lStyle] {\footnotesize{$v_\i$}};

\begin{scope}[xshift=3in]
\draw (0,0) 
node (v) {} -- (1,1) node(w1) {} -- (2,1) node[sStyle] {}
(v) -- (1,0) node(w2) {} -- (2,0) node[sStyle] {}
(v) -- (1,-1) node(w3) {} -- (2,-1) node[sStyle] {};
\draw (v) ++ (0,\off) node[lStyle] {\footnotesize{$v$}};
\draw (2,1) ++ (0,.8*\off) node[lStyle] {\footnotesize{$x_1$}};

\foreach \i in {1,...,3}
\draw (w\i) ++ (.3*\off,.8*\off) node[lStyle] {\footnotesize{$w_\i$}};
\end{scope}
\end{tikzpicture}
\caption{A 3-thread (left), as in Lemma~\ref{lem2}, and a 3-vertex $v$ with three 2-neighbors
(right), as in Lemma~\ref{lem3}.\label{lem2,3-fig}
Here and throughout Section~\ref{4list-reducibility-sec},
round vertices have all incident edges drawn, but square vertices have some
incident edges undrawn.}

\end{figure}

\begin{lem}
\label{lem3}
Let $v$ be a 3-vertex with 2-neighbors $w_1$, $w_2$, $w_3$, and let $x_1$ be the
other neighbor of $w_1$.  Let $G':=G-\{v,w_1,w_2,w_3\}$, and let
$\sigma'$ be a 14-good $L$-recoloring sequence that transforms
$\alpha_{\uphr G'}$ to $\beta_{\uphr G'}$.
If $\sigma'$ recolors $x_1$ at most 9 times, then $G$ has a
14-good $L$-recoloring sequence that transforms $\alpha$ to $\beta$.
\end{lem}
\begin{proof}
See Figure~\ref{lem2,3-fig}.
By Lemma~\ref{lem2}, we extend $\sigma'$ to a 14-good $L$-recoloring sequence for
$G-\{w_1\}$ that recolors $v$ at most 4 times.  After this, we extend the
14-good $L$-recoloring sequence to $w_1$ by the Key Lemma, since the number of
times its neighborhood is recolored is at most $4+9$.
\end{proof}

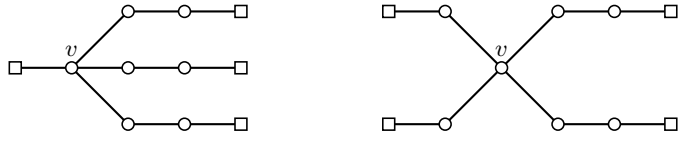
\begin{figure}[!h]
\centering
\begin{tikzpicture}[thick, scale=.75]
\tikzstyle{uStyle}=[shape = circle, minimum size = 4.5pt, inner sep = 0pt,
outer sep = 0pt, draw, fill=white, semithick]
\tikzstyle{sStyle}=[shape = rectangle, minimum size = 4.5pt, inner sep = 0pt,
outer sep = 0pt, draw, fill=white, semithick]
\tikzstyle{lStyle}=[shape = circle, minimum size = 4.5pt, inner sep = 0pt,
outer sep = 0pt, draw=none, fill=none]
\tikzset{every node/.style=uStyle}
\def\off{.4}

\draw (-1,0) node[sStyle] {} -- (0,0) 
node (v) {} -- (1,1) node {} -- (2,1) node {} -- (3,1) node[sStyle] {}
(v) -- (1,0) node {} -- (2,0) node {} -- (3,0) node[sStyle] {}
(v) -- (1,-1) node {} -- (2,-1) node {} -- (3,-1) node[sStyle] {};
\draw (v) ++ (0,.3) node[lStyle] {\footnotesize{$v$}};

\begin{scope}[xshift=3in]

\draw (-2,1) node[sStyle] {} -- (-1,1) node {} -- (0,0)
node (v) {} -- (1,1) node {} -- (2,1) node {} -- (3,1) node[sStyle] {};
\draw (-2,-1) node[sStyle] {} -- (-1,-1) node {} -- 
(v) -- (1,-1) node {} -- (2,-1) node {} -- (3,-1) node[sStyle] {};
\draw (v) ++ (0,.3) node[lStyle] {\footnotesize{$v$}};
\end{scope}

\end{tikzpicture}
\caption{%
A 4-vertex with 6 nearby 2-vertices;
two cases in Lemma~\ref{lem4}.\label{lem4-fig}}
\end{figure}

\begin{lem}
\label{lem4}
$G$ does not contain any 3-vertex with 4 or more nearby 2-vertices,
and $G$ does not contain any 4-vertex with 6 or more nearby 2-vertices.
\end{lem}
\begin{proof}
First suppose that $G$ contains a 3-vertex or 4-vertex $v$ that has
$d(v)-1$ incident 2-threads.  
For the case when $d(v)=4$, see Figure~\ref{lem4-fig} (left).
Form $G'$ from $G$ by deleting the interior
vertices of the $d(v)-1$ incident 2-threads.  And let $G'':=G'-v$.
By minimality, there exists a 14-good $L$-recoloring sequence $\sigma''$ for $G''$
that transforms $\alpha_{\uphr G''}$ to $\beta_{\uphr G''}$.
By the Key Lemma, we can extend $\sigma''$ to a 14-good $L$-recoloring sequence
$\sigma'$ for $G'$ that recolors $v$ at most $1+14/2=8$ times.
By applying Lemma~\ref{lem1} to each 2-thread
incident to $v$, we can extend $\sigma'$ to a 14-good $L$-recoloring sequence
$\sigma$ for $G$ that transforms $\alpha$ to $\beta$.

Suppose instead that $G$ contains a 3-vertex or 4-vertex $v$ that has 2 incident
1-threads and $d(v)-2$ incident 2-threads.
See Figure~\ref{lem4-fig} (right).
Form $G'$ from $G$ by deleting the interior vertices of the 2-threads incident
to $v$.  Form $G''$ from $G'$ by deleting $v$ and the interior vertices of the
1-threads incident to $v$.  Now the argument is very similar to that in the
previous paragraph.  By minimality, there exists a 14-good $L$-recoloring sequence
$\sigma''$ for $G''$ that transforms $\alpha_{\uphr G''}$ to $\beta_{\uphr
G''}$.  By Lemma~\ref{lem2}, we extend it to a 14-good $L$-recoloring sequence
$\sigma'$ for $G'$ that transforms $\alpha_{\uphr G'}$ to $\beta_{\uphr G'}$ and
recolors $v$ at most 4 times.  By applying Lemma~\ref{lem1} for each 2-thread
incident to $v$, we extend $\sigma'$ to a 14-good $L$-recoloring
sequence $\sigma$ for $G$ that transforms $\alpha$ to $\beta$.

Since $G$ has no 3-threads, by Lemma~\ref{lem2}, all possible cases are handled
above.
\end{proof}

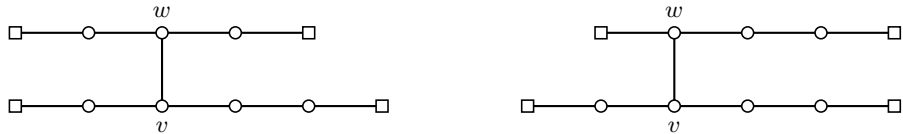
\begin{figure}[!h]
\centering
\begin{tikzpicture}[thick, scale=.975]
\tikzstyle{uStyle}=[shape = circle, minimum size = 4.5pt, inner sep = 0pt,
outer sep = 0pt, draw, fill=white, semithick]
\tikzstyle{sStyle}=[shape = rectangle, minimum size = 4.5pt, inner sep = 0pt,
outer sep = 0pt, draw, fill=white, semithick]
\tikzstyle{lStyle}=[shape = circle, minimum size = 4.5pt, inner sep = 0pt,
outer sep = 0pt, draw=none, fill=none]
\tikzset{every node/.style=uStyle}
\def\off{.285}

\draw (-2,0) node[sStyle] {} -- (-1,0) node {} -- (0,0) node (v) {} -- (1,0)
node {} -- (2,0) node {} -- (3,0) node[sStyle] {}
(v) -- (0,1) node (w) {} (-2,1) node[sStyle] {} -- (-1,1) node {} -- (w) --
(1,1) node {} -- (2,1) node[sStyle] {};
\draw (v) ++ (0,-\off) node[lStyle] {\footnotesize{$v$}};
\draw (w) ++ (0,\off) node[lStyle] {\footnotesize{$w$}};

\begin{scope}[xshift=2.75in]

\draw (-2,0) node[sStyle] {} -- (-1,0) node {} -- (0,0) node (v) {} -- (1,0)
node {} -- (2,0) node {} -- (3,0) node[sStyle] {}
(v) -- (0,1) node (w) {} (-1,1) node[sStyle] {}  -- (w) --
(1,1) node {} -- (2,1) node {} -- (3,1) node[sStyle] {};
\draw (v) ++ (0,-\off) node[lStyle] {\footnotesize{$v$}};
\draw (w) ++ (0,\off) node[lStyle] {\footnotesize{$w$}};

\end{scope}
\end{tikzpicture}
\caption{%
A $3_{2,1,0}$-vertex $v$ adjacent to a $3_{1,1,0}$-vertex $w$ (left)
and a $3_{2,1,0}$-vertex $v$ adjacent to a $3_{2,0,0}$-vertex $w$ (right);
two cases of Lemma~\ref{lem5}.\label{lem5-fig}}

\end{figure}

\begin{lem}
\label{lem5}
No $3_{2,1,0}$-vertex is adjacent to 
a $3_{1,1,0}$-vertex or a $3_{2,0,0}$-vertex or a $3_{2,1,0}$-vertex.
\end{lem}

\begin{proof}
Assume, to the contrary, that $G$ contains a $3_{2,1,0}$-vertex $v$ with a
3-neighbor $w$ that is a $3_{1,1,0}$-vertex or a $3_{2,0,0}$-vertex or a
$3_{2,1,0}$-vertex.  Figure~\ref{lem5-fig} shows the first two of these cases.
Form $G'$ from $G$ by deleting the interior vertices of a
2-thread incident to $v$.  Form $G''$ from $G'$ by deleting $v$ and the
interior vertex of an incident 1-thread.  Form $G'''$ from $G''$ by deleting
$w$ and all interior vertices of its incident threads.  By minimality, $G'''$
has a 14-good $L$-recoloring sequence $\sigma'''$ that transforms
$\alpha_{\uphr G'''}$ to $\beta_{\uphr G'''}$.  

If $w$ is a $3_{2,1,0}$-vertex or a $3_{1,1,0}$-vertex, then by
Lemma~\ref{lem2}, we extend $\sigma'''$ to a 14-good $L$-recoloring sequence
$\sigma''$ for $G''$ that recolors $w$ at most 4 times and transforms
$\alpha_{\uphr G''}$ to $\beta_{\uphr G''}$.  If instead $w$ is a
$3_{2,0,0}$-vertex, then we first extend to $w$ by the Key Lemma, recoloring
$w$ at most $14/2+1=8$ times, and then extend to its incident 2-thread by
Lemma~\ref{lem1}.  Again by Lemma~\ref{lem1}, we extend $\sigma''$ to a 14-good
$L$-recoloring sequence $\sigma'$ for $G'$ that recolors $v$ at most
$8+3=11$ times and transforms $\alpha_{\uphr G'}$ to $\beta_{\uphr
G'}$.  Now by Lemma~\ref{lem1}, we extend $\sigma'$ to a 14-good $L$-recoloring
sequence $\sigma$ for $G$ that transforms $\alpha$ to $\beta$.
Note that $\sigma$ recolors the neighbor of $v$ on its incident 2-thread at most
$11+3=14$ times.
\end{proof}

\begin{figure}[!h]
\centering
\begin{tikzpicture}[thick, scale=.625, yscale=.95]
\tikzstyle{uStyle}=[shape = circle, minimum size = 4.5pt, inner sep = 0pt,
outer sep = 0pt, draw, fill=white, semithick]
\tikzstyle{sStyle}=[shape = rectangle, minimum size = 4.5pt, inner sep = 0pt,
outer sep = 0pt, draw, fill=white, semithick]
\tikzstyle{lStyle}=[shape = circle, minimum size = 4.5pt, inner sep = 0pt,
outer sep = 0pt, draw=none, fill=none]
\tikzset{every node/.style=uStyle}
\def\off{.405}

\draw (-2,2) node[sStyle] {} -- (-1,1) node {} -- (0,0) node (v) {} 
(-2,-2) node[sStyle] {} -- (-1,-1) node {} -- (v) -- (1.5,0) node (x) {} -- (3,0)
node (w) {} 
(5,2) node[sStyle] {} -- (4,1) node {} -- (w)
(5,-2) node[sStyle] {} -- (4,-1) node {} -- (w);
\draw (v) ++ (0,\off) node[lStyle] {\footnotesize{$v$}};
\draw (w) ++ (0,\off) node[lStyle] {\footnotesize{$w$}};
\draw (x) ++ (0,\off) node[lStyle] {\footnotesize{$x$}};

\begin{scope}[xshift=3.75in]

\draw (-2,2) node[sStyle] {} -- (-1,1) node {} -- (0,0) node (v) {} 
(-2,-2) node[sStyle] {} -- (-1,-1) node {} -- (v) -- (1.5,0) node (x) {} -- (3,0)
node (w) {} 
(4,1) node[sStyle] {} -- (w)
(6,-2) node[sStyle] {} -- (5,-1.33) node {} -- (4,-.67) node {} -- (w);
\draw (v) ++ (0,\off) node[lStyle] {\footnotesize{$v$}};
\draw (w) ++ (0,\off) node[lStyle] {\footnotesize{$w$}};
\draw (x) ++ (0,\off) node[lStyle] {\footnotesize{$x$}};

\end{scope}
\end{tikzpicture}

\caption{A $3_{1,1,1}$-vertex $v$ weakly adjacent to a $3_{1,1,1}$-vertex $w$ (left)
and a $3_{1,1,1}$-vertex $v$ weakly adjacent to a $3_{2,1,0}$-vertex $w$
(right); the two cases of Lemma~\ref{lem6}\label{lem6-fig}.}
\end{figure}
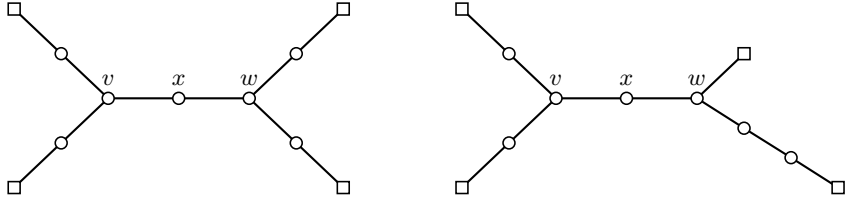

\begin{lem}
\label{lem6}
Each $3_{1,1,1}$-vertex has no weak 3-neighbor that is a $3_{2,1,0}$-vertex  
and no weak 3-neighbor that is a $3_{1,1,1}$-vertex.
\end{lem}
\begin{proof}
First suppose that $v$ and $w$ are both $3_{1,1,1}$-vertices and $v$ is a weak
neighbor of $w$; let $x$ be their common neighbor.  See Figure~\ref{lem6-fig}
(left).  Let $G':=G-x$ and form $G''$
from $G$ by deleting $v$, $w$, and all interior vertices of incident 1-threads.
By minimality, $G''$ has a 14-good $L$-recoloring sequence that transforms
$\alpha_{\uphr G''}$ to $\beta_{\uphr G''}$.  By applying Lemma~\ref{lem2}
twice, we extend $\sigma''$ to a 14-good $L$-recoloring sequence $\sigma'$ that
recolors each of $v$ and $w$ at most 4 times and transforms $\alpha_{\uphr G'}$
to $\beta_{\uphr G'}$.  By the Key Lemma, we extend $\sigma'$ to a 14-good
$L$-recoloring sequence for $G$ that transforms $\alpha$ to $\beta$.
The number of times that $\sigma$ recolors $x$ is at most $4+4+1$.

Now suppose instead that $v$ is a $3_{1,1,1}$-vertex and $w$ is a
$3_{2,1,0}$-vertex and that $v$ and $w$ have a common 2-neighbor, $x$.
See Figure~\ref{lem6-fig} (right).
The proof is nearly identical.  The difference is that to extend a 14-good
recoloring sequence from $G''$ to $G'$, we use Lemma~\ref{lem2} for $v$, but for
$w$ we first use the Key Lemma and then use Lemma~\ref{lem1}; note that $w$ is
recolored at most $1+14/2=8$ times.  Again, we can extend to a 14-good
$L$-recoloring sequence for $G$ that transforms $\alpha$ to $\beta$.  Now the
number of times that $\sigma$ recolors $x$ is at most $4+8+1=13$.
\end{proof}

\subsection{Discharging}
\label{4list-discharging-sec}

\begin{lem}
If $G$ is a graph with $\delta(G)\ge 2$ and $\mad(G)<22/9$, then $G$ contains
one of the configurations forbidden, by Lemmas~\ref{lem1}--\ref{lem6}, from
appearing in a minimal counterexample to Theorem~\ref{thm3}.  Thus,
Theorem~\ref{thm3} is true.
\end{lem}

\begin{proof}
Each vertex $v$ begins with charge $d(v)$ and we discharge so that each vertex
ends with charge at least $2+4/9$.  We use the following three discharging rules.

\begin{enumerate}
\item[(R1)] Each $3^+$-vertex sends $2/9$ to each nearby 2-vertex.
\item[(R2)] Each $4^+$-vertex and $3_{0,0,0}$-vertex and $3_{1,0,0}$-vertex
sends $1/9$ to each
3-neighbor and sends $1/18$ to each weak 3-neighbor with a common 2-neighbor.
\item[(R3)] Each $3_{1,1,0}$-vertex sends $1/18$ to each weak 3-neighbor with a
common 2-neighbor.
\end{enumerate}

Now we show that each vertex finishes with charge at least $2+4/9$.  
Recall, from Lemma~\ref{lem2}, that $G$ has no 3-threads.  
And recall, from Lemma~\ref{lem0}, that $\delta(G)\ge
2$.  If $d(v)=2$, then $\ch^*(v)=2+2(2/9)=2+4/9$.  If $d(v)\ge 5$, then
$\ch^*(v)\ge d(v)-2(2/9)d(v) = 5d(v)/9\ge 25/9 > 2+4/9$.  If $d(v)=4$, 
then, by Lemma~\ref{lem4}, $v$ has at most 5 nearby 2-vertices.
So $v$ gives away at most $5(2/9)$ by (R1) and at most $4(1/18)$ by (R2).
Now $\ch^*(v)\ge 4-5(2/9)-4(1/18)=36/9-12/9 = 24/9=2+6/9$.  Thus, 
by Lemma~\ref{lem4},
we assume that $v$ is a $3_{a,b,c}$-vertex, where $a+b+c\le 3$. 

If $v$ is a $3_{0,0,0}$-vertex, then $\ch^*(v)\ge 3-3(1/9) = 2+6/9$.

If $v$ is a $3_{1,0,0}$-vertex, then $\ch^*(v)\ge 3-2/9-2(1/9)-1/18=2+1/2$.

If $v$ is a $3_{1,1,0}$-vertex, then $\ch^*(v)\ge 3-2(2/9)-2(1/18)=2+4/9$.

If $v$ is a $3_{2,0,0}$-vertex, then $\ch^*(v)\ge 3-2(2/9)=2+5/9$.

If $v$ is a $3_{2,1,0}$-vertex, then $\ch^*(v)\ge 3-3(2/9)+1/9=2+4/9$, since
$v$ receives at least $1/9$ from its $3^+$-neighbor by (R2); this
is because, by Lemma~\ref{lem5},  the $3^+$-neighbor of $v$ cannot be a
$3_{1,1,0}$-vertex or a $3_{2,0,0}$-vertex or a $3_{2,1,0}$-vertex.

If $v$ is a $3_{1,1,1}$-vertex, then $\ch^*(v)\ge 3-3(2/9)+3(1/18) = 2+1/2$,
since $v$ receives at least $1/18$ from each weak $3^+$-neighbor by rule (R2)
or (R3); this is because, by Lemma~\ref{lem6}, a $3_{1,1,1}$-vertex
has no weak neighbor that is a $3_{1,1,1}$-vertex or a
$3_{2,1,0}$-vertex (or a $3_{2,0,0}$-vertex).
\end{proof}

\section*{Acknowledgment}
Thanks to Ben Moore for helpful comments on an earlier version of this paper.
And thanks to an anonymous referee for catching various inaccuracies and minor omissions.

\bibliographystyle{habbrv}
{\tiny{\bibliography{GraphColoring}}}
\end{document}